\documentclass[reqno,b5paper]{amsart}
\usepackage{amsmath}
\usepackage{amssymb}
\usepackage{amsthm}
\usepackage{color}
\usepackage{hyperref}
\usepackage{eqlist}
\usepackage{enumerate}
\usepackage[mathscr]{eucal}
\theoremstyle{plain}
\newtheorem{theorem}{Theorem}[section]
\newtheorem{prop}[theorem]{Proposition}
\newtheorem{lemma}{Lemma}[section]
\newtheorem{corol}{Corollary}[theorem]
\theoremstyle{definition}
\newtheorem{definition}{Definition}[section]
\newtheorem{remark}{\textnormal{\textbf{Remark}}}
\theoremstyle{remark}

\numberwithin{equation}{section}
\begin{document}
\title[On a class of $\alpha$-para Kenmotsu manifolds]%
{On a class of $\alpha$-para Kenmotsu manifolds}
\author[K. Srivastava \and S. K. Srivastava]
{K. Srivastava* \and S. K. Srivastava**}

\newcommand{\acr}{\newline\indent}

\address{\llap{*\,}Department of Mathematics\acr
                    D. D. U. Gorakhpur University\acr
                    Gorakhpur-273009\acr
                    Uttar Pradesh\acr
                    INDIA}
\email{ksriddu22@gmail.com}

\address{\llap{**\,}Department of Mathematics\acr
                   Central University of Himachal Pradesh\acr
                   Dharamshala-176215\acr
                   Himachal Pradesh\acr
	       INDIA}
\email{sachink.ddumath@gmail.com}


\begin{abstract}
The purpose of this paper is to classify $\alpha$-para Kenmotsu manifolds $M^3$ such that the projection of the image of concircular curvature tensor $L$ in one-dimensional linear subspace of $T_{p}(M^{3})$ generated by $\xi_{p}$ is zero. 
\end{abstract}

\maketitle

\section{Introduction}
The geometry of concircular transformations is a generalization of inversive geometry in the sense that the change of metric is more general than that induced by a circle preserving diffeomorphism \cite{mmt}. An interesting invariant of a concircular transformation is the concircular curvature tensor.

Let $(M^{2n+1},g)$ be a $(2n+1)$-dimensional connected pseudo-Riemannian manifold. The concircular curvature tensor $L$ \cite{yk1} of $M^{2n+1}$ is defined by
\begin{align}
L(X,Y)Z=&R(X,Y)Z-\frac{\tau}{2n(2n+1)}\bigl(g(Y,Z)X-g(X,Z)Y\bigr)\label{concir}
\end{align}
where $R$ is the curvature tensor, $\tau$ is the scalar curvature and $X, Y, Z\in\chi(M^{2n+1})$, $\chi(M^{2n+1})$ being the Lie algebra of vector fields of $M^{2n+1}$.

We observe immediately from the form of the concircular curvature tensor that pseudo-manifolds with vanishing concircular curvature tensor are of constant curvature. Thus one can think of the concircular curvature tensor as a measure of the failure of a pseudo-Riemannian manifold to be of constant curvature (see also \cite{debmmt}). This paper is organized as follows: In \S 2 The basic information about almost paracontact metric manifolds, normal almost paracontact  metric manifolds and the curvature tensor of the manifolds are given. In \S 3 we have obtained the relation between second order parallel tensor and the associated metric on  $\alpha$-para Kenmotsu manifold . In \S 4 we found the necessary and sufficient condition for an $\alpha$-para Kenmotsu manifold to be $\xi$-concirularly flat. Finally, we cited of an $\alpha$-para Kenmotsu manifold in \S 5.

\section{Preliminaries}
\subsection{Almost paracontact metric manifolds}
A $C^{\infty}$ smooth manifold $M^{2n+1}$ of dimension $(2n+1)$, is said to have triplet $(\phi, \xi, \eta)-$structure, if it admits an endomorphism $\phi$, a unique vector field $\xi$  and a contact form $\eta$ satisfying: 
\begin{align}\label{eta}
\phi^2 = I -\eta\otimes\xi\,\,\, {\rm\and}\,\,\,\,  \eta\left(\xi\right)=1
\end{align}
where $I$ is the identity transformation; and the endomorphism $\phi$ induces an almost paracomplex structure on each fibre of $D=ker(\eta),$ the contact subbundle, i.e., eigen distributions  $D^{\pm 1}$ corresponding to the characteristic values $\pm 1$ of $\phi$ have equal dimension $n.$\\
From the equation (\ref{eta}), it can be easily deduce that 
\begin{eqnarray}\label{phixi}
\phi\xi = 0, && \eta o\phi = 0 \,\,\,\,{\rm and \,\,\,\, rank}(\phi) = 2n.
\end{eqnarray}
This triplet structure$-(\phi, \xi, \eta)$ is called an almost paracontact structure and the manifold $M^{2n+1}$ equipped with the $(\phi, \xi, \eta)-$structure is called an almost paracontact manifold \cite{skm}. If an almost paracontact manifold admits a pseudo-Riemannian metric \cite{sz}, $g$ satisfying: 
\begin{align}\label{gphi}
g\left(\phi X, \phi Y\right) = -g(X, Y) + \eta(X)\eta(Y)
\end{align}
where signature of $g$ is necessarily $(n+1, n)$ for any vector fields  $X$ and $Y$; then the quadruple$-(\phi, \xi, \eta, g)$ is called an almost paracontact metric structure and the manifold $M^{2n+1}$ equipped with paracontact metric structure is called an almost paracontact metric manifold. With respect to $g$, $\eta$ is metrically dual to $\xi$, that is
\begin{align}\label{gx}
g(X, \xi) = \eta (X) 
\end{align}
Also, equation (\ref{gphi}) implies that
\begin{align}\label{gphix}
g(\phi X, Y) = -g(X, \phi Y).
\end{align}
 Further, in addition to the above properties, if the structure $-(\phi, \xi, \eta, g)$ satisfies:
\begin{align}
d\eta(X, Y)= g(X, \phi Y),
\nonumber
\end{align} 
for all vector fields $X$, $Y$ on $M^{2n+1}$, then the manifold is called a paracontact metric manifold and the corresponding structure$-(\phi, \xi, \eta, g)$ is called a paracontact structure with the associated metric $g$ \cite{sz}. For an almost paracontact metric manifold, there always exists a special kind of local pseudo-orthonormal basis $\left\{X_{i}, X_{i^*}, \xi\right\}$;  where $X_{i^*}=\phi X_{i};$ $\xi$ and $X_{i}$'s are space-like vector fields and $X_{i^*}$'s  are time-like. Such a basis is called $\phi -$basis. Hence, an almost paracontact metric manifold $(M^{2n+1},\phi, \xi, \eta, g)$ is an odd dimensional manifold with a structure group $\mathbb{U}(n,\mathbb{R})\times Id$, where $\mathbb{U}(n,\mathbb{R})$ is the para-unitary group isomorphic to $\mathbb{G}\mathbb{L}(n,\mathbb{R})$. 
\subsection{Normal almost paracontact metric manifolds}  
On an almost paracontact manifold, one defines the (2, 1)-tensor field $N_\phi$ by
\begin{align}\label{n}
N_{\phi}:=[\phi, \phi]-2d\eta\otimes\xi,
\end{align}
where $[\phi, \phi]$ is the Nijenhuis torsion of $\phi$. If $N_\phi$ vanishes identically, then we say that the manifold $M^{2n+1}$ is a normal almost paracontact metric manifold (\cite{sk,sz}). The normality condition implies that the almost paracomplex structure $J$ defined on $M^{2n+1}\times\mathbb{R}$ by
\begin{align}
J\Bigg(X,\lambda\frac{d}{dt}\Bigg)=\Bigg(\phi X+\lambda\xi,\eta(X)\frac{d}{dt}\Bigg)\nonumber
\end{align}
is integrable. Here $X$ is tangent to $M^{2n+1}$, $t$ is the coordinate on $\mathbb{R}$ and $\lambda$ is $C^\infty$ function on $M^{2n+1}\times\mathbb{R}$.   
Now we recall the following proposition which characterized the normality of almost paracontact metric 3-manifolds:
\begin{prop}\cite{jwleg} For almost paracontact metric 3-manifold $M^3$,the following three conditions are mutually equivalent
\begin{itemize}
 \item[(i)] $M^3$ is normal,
\item[(ii)] there exist smooth functions $\alpha$, $\beta$ on $M^3$ such that
\begin{eqnarray}\label{nablaxphiy}
(\nabla_{X}\phi)Y = \beta(g(X,Y)\xi-\eta(Y)X)+\alpha (g(\phi X, Y)\xi-\eta(Y)\phi X),
\end{eqnarray}
\item[(iii)] there exist smooth functions $\alpha$, $\beta$ on $M^3$ such that
\begin{eqnarray}\label{nablaxxi}
\nabla_{X}\xi = \alpha (X - \eta(X)\xi) + \beta \phi X\end{eqnarray}
\end{itemize}
where $\nabla$ is the Levi-Civita connection of the pseudo-Riemannian metric $g$. 
\end{prop}
The functions $\alpha, \beta$ appearing in (\ref{nablaxphiy}) and (\ref{nablaxxi}) are given by    
\begin{eqnarray}
2\alpha = trace\left\{X\rightarrow\nabla_{X}\xi\right\},& 2\beta = trace\left\{X\rightarrow\phi\nabla_{X}\xi\right\}.
\end{eqnarray}

\begin{definition}
A normal almost paracontact metric 3-manifold is called
\begin{itemize}
\item [$\bullet$] paracosymplectic if $\alpha=\beta=0$ \cite{pd},
\item [$\bullet$] quasi-para Sasakian if and only if  $\alpha=0$ and $\beta\ne0$ \cite{es},
\item [$\bullet$] $\beta$-para Sasakian if and only if  $\alpha=0$ and $\beta$ is non-zero constant, in particular para Sasakian if $\beta=-1$ \cite{sz},
\item [$\bullet$] $\alpha$-para Kenmotsu if $\alpha$ is non-zero constant and $\beta=0$ \cite{jwslant}.
\end{itemize}
\end{definition}
\medskip
\subsection{Curvature properties of normal almost paracontact metric 3-manifolds}
In $(2n+1)$-dimensional connected pseudo-Riemannian manifold $(M^{2n+1},g)$, the curvature tensor  $R$ \cite{neill} and the projective curvature tensor $P$ \cite{ykm} are defined by
\begin{align}
R(X,Y)Z=&\nabla_{[X,Y]}Z-[\nabla_{X},\nabla_{Y}]Z,\label{cur}\\
P(X,Y)Z=&R(X,Y)Z-\frac{1}{2n}\bigl(g(SY,Z)X-g(SX,Z)Y\bigr).\label{pxyz}
\end{align} 
In $3$-dimensional pseudo-Riemannian manifold the curvature tensor satisfies \cite{de1990}:  
\begin{align}\label{rxyz}
\tilde{R}(X,Y, Z, W)=&g(X,W)g(SY,Z)-g(X,Z)g(SY, W)+g(Y,Z)g(SX, W)\nonumber\\
-&g(Y,W)g(SX, Z)-\frac{\tau}{2}\left\{g(Y,Z)g(X, W)-g(X,Z)g(Y, W)\right\}
\end{align}
where $\tilde{R}(X,Y,Z,W)=g(R(X,Y)Z,W),$  $\tau=trace(S)$ is the scalar curvature of the manifold and Ricci operator $S$ is defined by
\begin{equation}
g(SX,Y)=Ric(X,Y).
\end{equation}

Using (\ref{eta}), (\ref{phixi}), (\ref{gphi}), (\ref{gphix}), (\ref{nablaxxi}),(\ref{cur}) and (\ref{rxyz}) it is easy to prove the following lemma:
\begin{lemma}
Let $M^{3}(\phi, \xi, \eta, g)$ be a normal almost paracontact metric manifold, then we have
\begin{align}
& R(X,Y)\xi = \bigl\{(Y\alpha) +(\alpha^{2}+\beta^{2})\eta(Y)\bigr\}\phi^{2}X-\bigl\{(X\alpha) +(\alpha^{2}+\beta^{2})\eta(X)\bigr\}\phi^{2}Y\label{4.3}\nonumber \\
&\hspace{1.5cm}+\bigl\{(Y\beta) +2\alpha\beta\eta(Y)\bigr\}\phi X-\bigl\{(X\beta) +2\alpha\beta\eta(X)\bigr\}\phi Y\\
&Ric(X,Y)=\bigl\{\frac{\tau}{2}-(\xi\alpha)-\alpha^2-\beta^2\bigr\}g(X,Y)+\bigl\{(\xi\alpha)+3(\alpha^2+\beta^2)-\frac{\tau}{2}\bigr\}\eta(X)\eta(Y)\label{4.4}\nonumber \\
&\hspace{1.5cm}+ \{\eta(Y)(X\alpha) +\eta(X)(Y\alpha)\}-\{\eta(Y)\phi (X\beta) +\eta (X)\phi (Y\beta)\},\\
& Ric(Y,\xi)=(Y\alpha) -\phi (Y\beta)+\bigl\{(\xi\alpha) + 2(\alpha^2 +\beta^2)\bigr\}\eta(Y),\label{4.5}\\
&\xi\beta +2\alpha\beta =0.\label{4.6}
\end{align}
\end{lemma}
\begin{remark} From (\ref{4.6}), it follows that if  $\alpha, \beta=$constant, then $M^3$ is either $\beta$-para Sasakian or $\alpha$-para Kenmotsu or paracosymplectic.\end{remark}
\begin{prop}\label{cat}
Let $M^{3}(\phi, \xi, \eta, g)$ be an $\alpha$-para Kenmotsu manifold, then we have 
\begin{align}\label{3r}
R(X,Y)Z=&\bigl({\tau}/{2}-2\alpha^2\bigr)(g(Y, Z)X-g(X,Z)Y)\nonumber\\
-& ({\tau}/{2}-3\alpha^2)\{g(Y,Z)\eta(X)-g(X,Z)\eta(Y)\}\xi\nonumber\\
+&({\tau}/{2}-3\alpha^2)\{Y\eta(X)-X\eta(Y)\}\eta(Z).
\end{align}
\end{prop}
\begin{proof}
In view of (\ref{rxyz}) and (\ref{4.3}), we have (\ref{3r}).
\end{proof}

\section{Second order parallel tensor field}
\begin{definition}
A tensor $T$ of second order is said to be a second order parallel tensor if $\nabla T=0,$ where $\nabla$ denotes the operator of covariant differentiation with respect to the associated metric $g$ \cite{de}.
\end{definition}
Now we give the following result which established the relation between second order parallel tensor and the associated metric on  $\alpha$-para Kenmotsu manifold:
\begin{theorem}\label{ricciparallel}
On an $\alpha$-para Kenmotsu manifold $M^3$ a second order parallel tensor is a constant multiple of the associated tensor.
\end{theorem}
\begin{proof}
Let $h$ denotes symmetric $(0,2)$-tensor field on $M^3$ such that $\nabla h=0.$ Then it follows that
\begin{align}\label{par1}
h(R(Z,X)Y,W)+h(Y,R(Z,X)W)=0
\end{align}
for any $X, Y, Z, W\in\Gamma(TM^3).$ Substituting $Y=Z=W=\xi$ in (\ref{par1}), we obtain
\begin{align}
h(R(\xi, X)\xi, \xi)=0\nonumber
\end{align}
which gives by virtue of (\ref{3r}) that
\begin{align}\label{par2}
h(X,\xi)=\eta(X)h(\xi,\xi).
\end{align}
 Differentiating (\ref{par2}) along $Y$ and using (\ref{nablaxxi}) and (\ref{par2}), we have 
\begin{align}\label{par3}
h(X,Y)=h(\xi,\xi)g(X, Y).
\end{align}
Again differentiating (\ref{par3}) covariantly along any vector field on $M^3$ it can be easily seen that $h(\xi,\xi)$ is constant. This completes the proof.
\end{proof}
\begin{remark}
If the Ricci tensor field is parallel in an  $\alpha$-para Kenmotsu manifold $M^3$, then it is an Einstein manifold.
\end{remark} 
Let us suppose that $h$ is a parallel $2$-form on $M^3$, that is, 
\begin{align}\label{hxy}
 h(X,Y)=-h(Y,X) \,\,\,{\rm and}\,\,\, \nabla h=0.
\end{align}
Then we prove the following result:
\begin{theorem}
Let $M^{3}(\phi, \xi, \eta, g)$ be an $\alpha$-para Kenmotsu manifold. Then non-zero parallel 2-form cannot occur on $M^3$.
\end{theorem} 
\begin{proof}
For the parallel $2$-form, we have from (\ref{hxy})
\begin{align}\label{hxy1}
h(\xi,\xi)=0.
\end{align}
Differentiating (\ref{hxy1}) covariantly along $X$ and applying (\ref{nablaxxi}) and (\ref{hxy1}), we have
\begin{align}\label{hxy2}
h(X,\xi)=0.
\end{align}
Further differentiating above covariantly with respect to $Y$ and using (\ref{nablaxxi}) and (\ref{hxy2}) yields
\begin{align}
h(X,Y)=0.\nonumber
\end{align}
This completes the proof.
\end{proof}
\section{Main results}
$\xi$-conformally flat $K$-contact manifolds have been studied by Zhen $et\,\, al.$\cite{zg}. Since at each point $p\in M^{2n+1}$ the tangent space $T_{p}(M^{2n+1})$ can be decomposed into the direct sum $T_{p}(M^{2n+1})=\phi(T_{p}(M^{2n+1}))\oplus\{\xi_{p}\},$ where $\{\xi_{p}\}$ is the one-dimensional linear subspace of $T_{p}(M^{2n+1})$ generated by $\xi_{p}$, the conformal curvature tensor $C$ is a map 
\begin{align}
C:T_{p}(M^{2n+1})\times T_{p}(M^{2n+1})\times T_{p}(M^{2n+1})\rightarrow\phi(T_{p}(M^{2n+1}))\oplus\{\xi_{p}\}.\nonumber
\end{align}
An almost contact metric manifold $M^{2n+1}$ is called $\xi$-conformally flat  if the projection of the image of $C$ in $\{\xi_{p}\}$ is zero \cite{zg}.

Analogous to the definition of $\xi$-conformally flat almost contact metric manifold we define $\xi$-concirularly flat normal almost paracontact metric manifold. 

\begin{definition}
A normal almost paracontact metric manifold $M^{2n+1}$ is called $\xi$-concirularly flat  if the condition $L(X,Y)\xi=0$ holds on $M^{2n+1}$, where concircular curvature tensor $L$ is defined by (\ref{concir}).
\end{definition}
\begin{theorem}\label{main}
Let $M^{3}(\phi, \xi, \eta, g)$ be an $\alpha$-para Kenmotsu manifold. Then $M^3$ is $\xi$-concirularly flat if and only if the scalar curvature $\tau=6\alpha^2$.
\end{theorem}
\begin{proof} Putting $Z=\xi$ in (\ref{concir}) and using (\ref{4.3}) and (\ref{4.5}), we have
\begin{align}\label{confor1}
L(X,Y)\xi=(\alpha^2-\tau /6)\{\eta(Y)X-\eta(X)Y\}.
\end{align}
This implies that $L(X,Y)\xi=0$ if and only if $\tau=6\alpha^2$.\\ 
\end{proof}
As a corollary of the above theorem we have the following result:
\begin{corol}
Let $M^{3}(\phi, \xi, \eta, g)$ be a $\xi$-concirularly flat $\alpha$-para Kenmotsu manifold. Then 
\begin{itemize}
\item[(i)]\hspace{1cm} Ricci is parallel. 
\item[(ii)]\hspace{1cm} $M^3$ is Einstein.
\end{itemize}
\end{corol}
\begin{remark}An $\alpha$-para Kenmotsu manifold $M^{3}(\phi, \xi, \eta, g)$ is $\xi$-projectively flat.\end{remark}
\section{Example}
We consider the 3-dimensional manifold $M^3=\mathbb{R}^2\times\mathbb{R}_{-}\subset\mathbb{R}^3$ with the standard cartesian coordinates $(x,y,z)$. Define the almost paracontact structure $(\phi,\xi,\eta)$ on $M^3$ by
\begin{eqnarray}\label{e1}
\phi e_{1}=e_{2}, &\phi e_{2}=e_{1}, &\phi e_{3}=0,\,\,\xi =e_{3},\,\, \eta =dz,
\end{eqnarray}   
where ${e_{1}}=\frac{\partial}{\partial x}$, ${e_{2}}=\frac{\partial}{\partial y}$ and ${e_{3}}=\frac{\partial}{\partial z}$.  By straightforward calculations, one verifies that
\begin{eqnarray}
[\phi,\phi ]({e_{i}},{e_{j}})-2d\eta({e_{i}},{e_{j}})=0,&1\le i<j\le 3,\nonumber
\end{eqnarray}
which implies that the structure is normal. Let $g$ be the pseudo-Riemannian metric defined by 
\begin{align}\label{metric}
[g(e_{i},e_{j})]=
\begin{bmatrix}
exp(2z) & 0             &     0\\
0         &-exp(2z)     &     0\\
0         & 0               &      1
\end{bmatrix}.
\end{align}
For the Levi-Civita connection, we obtain 
\begin{align}\label{levi}
\begin{matrix}
\nabla_{e_{1}}{e_{1}}=-exp(2z)e_{3},            &    \nabla_{e_{1}}{e_{2}}=0,                  &\nabla_{e_{1}}{e_{3}}=e_{1},\\
\nabla_{e_{2}}{e_{1}}=0,              &    \nabla_{e_{2}}{e_{2}}=exp(2z)e_{3},                &\nabla_{e_{2}}{e_{3}}=e_{2},\\
\nabla_{e_{3}}{e_{1}}=e_{1},              &    \nabla_{e_{3}}{e_{2}}=e_{2},                &\nabla_{e_{3}}{e_{3}}=0.
\end{matrix}
\end{align}

Using the above expressions and (\ref{nablaxxi}), we find $\alpha=1,\,\,\beta=0$. Hence the manifold is an $\alpha$-para Kenmotsu manifold. With the help of (\ref{levi}) and (\ref{cur}), we have 
\begin{align}\label{curvature}
\begin{matrix}
R({e_{1}},{e_{2}})e_{3}=R({e_{1}},{e_{3}})e_{1}=-exp(2z)e_{3},&R({e_{1}},{e_{2}})e_{2}=-exp(2z)e_{1},\\
R({e_{1}},{e_{2}})e_{1}=-exp(2z)e_{2},&R({e_{2}},{e_{3}})e_{2}=exp(2z)e_{3},\\
R({e_{1}},{e_{3}})e_{3}=e_{1},&R({e_{2}},{e_{3}})e_{3}=e_{2},\\
R({e_{2}},{e_{3}})e_{1}=0,& R({e_{1}},{e_{3}})e_{2}=0.\\
\end{matrix}
\end{align}
From (\ref{curvature}) it is not hard to see that the scalar curvature $\tau=6$. Therefore $M^3$ is $\xi$-concircularly flat. Thus theorem \ref{main} is verified.

\end{document}